\documentclass[preprint,12pt]{elsarticle}

\usepackage[utf8]{inputenc}
\usepackage[T1]{fontenc}
\usepackage{lmodern}
\usepackage{microtype}
\usepackage[a4paper,margin=1in]{geometry}
\usepackage{amsmath,amssymb,amsthm,mathtools}
\usepackage{enumitem}
\usepackage{booktabs}
\usepackage{xcolor}
\usepackage{tikz}
\usetikzlibrary{arrows.meta,positioning,calc,fit,backgrounds,decorations.pathreplacing,shapes.geometric}
\usepackage[colorlinks=true,linkcolor=blue!60!black,citecolor=blue!60!black,urlcolor=blue!60!black]{hyperref}

\newtheorem{theorem}{Theorem}[section]
\newtheorem{lemma}[theorem]{Lemma}

\newtheorem{corollary}[theorem]{Corollary}

\theoremstyle{definition}
\newtheorem{definition}[theorem]{Definition}

\theoremstyle{remark}
\newtheorem{remark}[theorem]{Remark}
\newtheorem*{problemHLMP}{Problem 4.1 (Hoang--Levit--Mandrescu--Pham)}
\newtheorem*{conjHLMP}{Conjecture 4.2 (Hoang--Levit--Mandrescu--Pham)}

\newcommand{\W}{\mathbf W}

\newcommand{\Ind}{\operatorname{Ind}}

\begin{document}
	
	\begin{frontmatter}
		
		\title{The $2$-Quasi-Regularizability Conjecture and Independence Polynomials of $\W_p$ Graphs}

		\author[DEPTO]{Kevin Pereyra}\ead{kdpereyra@unsl.edu.ar}
		\address[DEPTO]{Departamento de Matem\'atica, Universidad Nacional de San Luis, San Luis, Argentina}
		
		\begin{abstract}
Hoang, Levit, Mandrescu and Pham asked for structural conditions ensuring that the independence polynomial of a $\W_p$ graph is log-concave, or at least unimodal, and conjectured that a connected $\W_2$ graph is $2$-quasi-regularizable if and only if $n(G)\ge 3\alpha(G)$ (2026).  We prove the conjecture.  The key point is a local expansion theorem: if $G$ is connected and belongs to $\W_2$, then every non-maximum independent set $A$ satisfies
\[
|N_G(A)|\ge 2|A|.
\]
Thus the only possible obstruction to $2$-quasi-regularizability in a connected $\W_2$ graph comes from maximum independent sets, where the condition is exactly $n(G)-\alpha(G)\ge 2\alpha(G)$.

We also give coefficient criteria for log-concavity and unimodality of independence polynomials of $\W_p$ graphs.  These criteria combine the standard two-sided coefficient inequalities, as collected by Hoang--Levit--Mandrescu--Pham, with $\lambda$-quasi-regularizability.  The new $p=2$ threshold proved here inserts the missing case into the same framework and yields explicit log-concavity and unimodality regions for connected $\W_2$ graphs.
		\end{abstract}
		
		\begin{keyword}
			well-covered graph; $1$-well-covered graph; $\W_p$ graph; quasi-regularizable graph; independence polynomial; log-concavity; unimodality.
			\MSC[2020] 05C69; 05C30; 05C31; 05C35.
		\end{keyword}
		
	\end{frontmatter}

\section{Introduction}

Let $G$ be a finite simple graph.  Its independence polynomial, introduced by Gutman and Harary \cite{GutmanHarary}, is
\[
        I(G;x)=\sum_{k=0}^{\alpha(G)}s_k x^k,
\]
where $s_k$ denotes the number of independent sets of cardinality $k$ in $G$.  The sequence $(s_0,s_1,\ldots,s_{\alpha(G)})$ is a natural enumerative invariant of $G$, but it is far from being constrained in general: Alavi, Malde, Schwenk and Erdos showed that prescribed order patterns can occur among its coefficients \cite{AlaviMaldeSchwenkErdos}.  Nevertheless, important graph classes do force regular behavior.  For instance, Heilmann and Lieb's real-rootedness theorem for matching polynomials \cite{HeilmannLieb}, and Chudnovsky and Seymour's real-rootedness theorem for independence polynomials of claw-free graphs \cite{ChudnovskySeymour}, both imply log-concavity in their respective settings.  Root locations and coefficient questions for independence polynomials have been studied from several viewpoints; see, for example, Brown--Hickman--Nowakowski \cite{BrownHickmanNowakowski}, the survey of Levit--Mandrescu \cite{LevitMandrescuSurvey}, and the later work on trees and clique-cover products \cite{Zhu,ZhuChen,KadrawiLevit}.  The broader role of log-concavity in combinatorics is surveyed by Stanley \cite{StanleyLogConcave}, and the modern theory is also influenced by Huh's work on chromatic and matroidal log-concavity \cite{HuhChromatic,HuhKatz}.

This paper concerns the class $\W_p$ introduced by Staples \cite{StaplesThesis,StaplesJGT}.  Well-covered graphs originate in the work of Plummer \cite{Plummer1970,PlummerSurvey}; very well-covered graphs were introduced by Favaron \cite{Favaron}, and structural results for special well-covered families include \cite{FinbowHartnellNowakowski,ToppVolkmann}.  A graph is in $\W_p$ if every $p$ pairwise disjoint independent sets can be enlarged to $p$ pairwise disjoint maximum independent sets.  The class $\W_1$ is precisely the class of well-covered graphs.  The class $\W_2$ is especially important because it coincides, for graphs without isolated vertices, with the class of $1$-well-covered graphs \cite{StaplesJGT,LevitMandrescuRevisited}; related $\W_2$ subclasses were studied by Pinter \cite{PinterThesis}.  These graph classes also occur naturally in edge-ideal theory and Gorenstein graph theory \cite{Villarreal,HoangTrung}.  On the enumerative side, counterexamples, roller-coaster phenomena and positive results for very well-covered graphs appear in \cite{MichaelTraves,LevitMandrescuWellCoveredCounterexamples,LevitMandrescuRoller,LevitMandrescuVeryWellCovered,ChenWang}.  Corona constructions, introduced by Frucht and Harary \cite{FruchtHarary}, provide a large source of well-covered and $\W_p$ examples.  In a recent paper, Hoang, Levit, Mandrescu and Pham studied quasi-regularizability and log-concavity for $\W_p$ graphs and proved the $p$-quasi-regularizability threshold for connected $\W_p$ graphs when $p\ne2$ \cite{HoangLevitMandrescuPham}.  They left the following two questions.

\begin{problemHLMP}
What conditions on the $\W_p$ graph $G$ guarantee that the independence polynomial $I(G;x)$ is log-concave or at least unimodal?
\end{problemHLMP}

\begin{conjHLMP}
Let $G$ be a connected $\W_2$ graph.  Then $G$ is $2$-quasi-regularizable if and only if
\[
        n(G)\ge 3\alpha(G).
\]
\end{conjHLMP}

The exclusion of $p=2$ in the known threshold is not cosmetic.  The cycle $C_5$ belongs to $\W_2$ but is not $2$-quasi-regularizable: if $S$ is a maximum independent set, then $|S|=2$ and $|N(S)|=3<4$.  This is exactly the failure of the inequality $n(C_5)\ge3\alpha(C_5)$; see Figure~\ref{fig:C5}.  The main result below proves that this is the only kind of failure.

\begin{center}\begin{minipage}{0.92\textwidth}\noindent\textbf{Main structural theorem.}\ \itshape
Let $G$ be a connected graph in $\W_2$.  If $A$ is an independent set which is not maximum, then
\[
        |N_G(A)|\ge2|A|.
\]
Consequently, a connected $\W_2$ graph is $2$-quasi-regularizable if and only if $n(G)\ge3\alpha(G)$.
\end{minipage}\end{center}

The proof is local.  If a smallest counterexample $A$ existed, then $|N(A)|=2|A|-1$.  Each vertex of $A$ would have a unique private neighbor; all other neighbors would see at least two vertices of $A$; private neighbors would be forced away from the localization $G_A$; every non-private vertex touching $G_A$ would be universal on $A$; and finally a second localization would build an augmenting set that contradicts the $\W_2$ avoidance property.  Figures~\ref{fig:normal-form}--\ref{fig:final-obstruction} show the successive reductions.

The local expansion theorem and its application to Conjecture 4.2 are the new structural part of the paper.  The coefficient criteria in Section~\ref{sec:coefficients} are included to make the consequence for Problem 4.1 self-contained: they combine known two-sided coefficient inequalities with the quasi-regularizability input, and in the case $\lambda=p$ recover the same discriminant range used in \cite[Theorem~3.3]{HoangLevitMandrescuPham}.

For Problem 4.1 we give a flexible coefficient answer.  If $G$ is connected, $G\in\W_p$, $n=n(G)$, $\alpha=\alpha(G)$, and $G$ is $\lambda$-quasi-regularizable, then the inequalities
\[
        (k+1)s_{k+1}\le\bigl(n-(\lambda+1)k\bigr)s_k,
        \qquad
        p(\alpha-k)s_k\le(k+1)s_{k+1}
\]
lead to explicit log-concavity and unimodality criteria.  The new $p=2$ threshold supplies the missing quasi-regularizability input exactly when $n\ge3\alpha$.

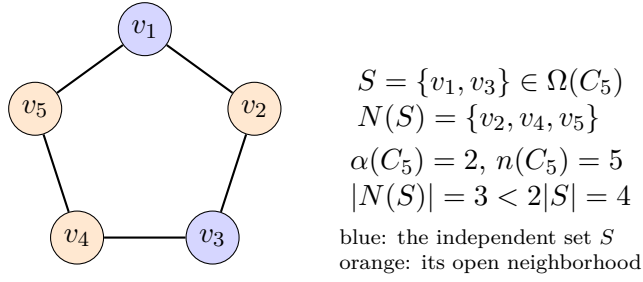
\begin{figure}[t]
\centering
\begin{tikzpicture}[scale=1.05, every node/.style={circle,draw,minimum size=7mm,inner sep=0pt,font=\small}]
  \node[fill=blue!16]   (v1) at (90:1.45)  {$v_1$};
  \node[fill=orange!18] (v2) at (18:1.45)  {$v_2$};
  \node[fill=blue!16]   (v3) at (-54:1.45) {$v_3$};
  \node[fill=orange!18] (v4) at (-126:1.45){$v_4$};
  \node[fill=orange!18] (v5) at (162:1.45) {$v_5$};
  \foreach \i/\j in {1/2,2/3,3/4,4/5,5/1}{\draw[thick] (v\i)--(v\j);}
  \node[draw=none,rectangle,align=left,font=\small] at (4.35,0.55)
       {$S=\{v_1,v_3\}\in\Omega(C_5)$\\$N(S)=\{v_2,v_4,v_5\}$};
  \node[draw=none,rectangle,align=left,font=\small] at (4.35,-0.45)
       {$\alpha(C_5)=2$, $n(C_5)=5$\\$|N(S)|=3<2|S|=4$};
  \node[draw=none,rectangle,align=left,font=\scriptsize] at (4.35,-1.35)
       {blue: the independent set $S$\\orange: its open neighborhood};
\end{tikzpicture}
\caption{The graph $C_5$ is a connected $\W_2$ graph, but it is not $2$-quasi-regularizable.  The obstruction is a maximum independent set, and $5<3\cdot2$.}
\label{fig:C5}
\end{figure}

The paper is organized as follows.  Section~\ref{sec:preliminaries} fixes the notation and records the standard tools on independence polynomials, quasi-regularizability and \(\W_2\) graphs.  Section~\ref{sec:coefficients} proves the coefficient criteria used to answer Problem 4.1.  Section~\ref{sec:local} proves the local expansion theorem for connected \(\W_2\) graphs, and Section~\ref{sec:conjecture} applies it to Conjecture 4.2.  Section~\ref{sec:consequences} derives the corresponding log-concavity and unimodality consequences, and the last section closes with brief remarks.

\section{Preliminaries}\label{sec:preliminaries}

All graphs considered in this paper are finite, undirected, and simple.  For undefined graph-theoretic terminology we refer the reader to Diestel~\cite{Diestel}; for matching terminology we refer to Lov\'asz and Plummer~\cite{LovaszPlummer}.

We write $V(G)$ and $E(G)$ for the vertex set and the edge set, $n(G)=|V(G)|$, and $\alpha(G)$ for the independence number.  The family of independent sets is denoted by $\Ind(G)$, and the family of maximum independent sets by $\Omega(G)$.  For $A\subseteq V(G)$, the open neighborhood is
\[
        N_G(A)=\{v\in V(G)-A:\ uv\in E(G)\text{ for some }u\in A\},
\]
with $N_G[A]=A\cup N_G(A)$.  When $A$ is independent, the localization of $G$ at $A$ is
\[
        G_A=G-N_G[A].
\]
For $A=\{v\}$ we write $G_v$.  Figure~\ref{fig:localization-example} illustrates the localization operation.

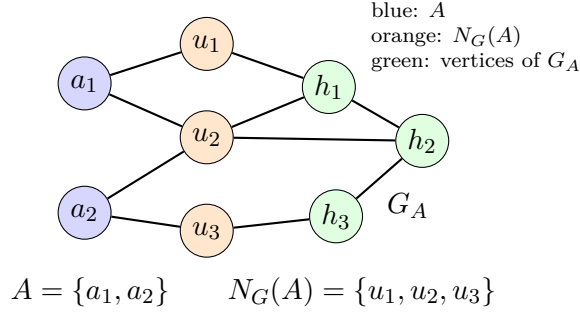
\begin{figure}[t]
\centering
\begin{tikzpicture}[scale=0.95,font=\small,
  every node/.style={circle,draw,minimum size=7mm,inner sep=0pt}]
  \node[fill=blue!16] (a1) at (0,0.9) {$a_1$};
  \node[fill=blue!16] (a2) at (0,-0.9) {$a_2$};
  \node[fill=orange!20] (n1) at (1.7,1.45) {$u_1$};
  \node[fill=orange!20] (n2) at (1.7,0.15) {$u_2$};
  \node[fill=orange!20] (n3) at (1.7,-1.15) {$u_3$};
  \node[fill=green!12] (h1) at (3.4,0.85) {$h_1$};
  \node[fill=green!12] (h2) at (4.7,0.1) {$h_2$};
  \node[fill=green!12] (h3) at (3.5,-0.95) {$h_3$};
  \draw[thick] (a1)--(n1) (a1)--(n2) (a2)--(n2) (a2)--(n3);
  \draw[thick] (n1)--(h1) (n2)--(h1) (n2)--(h2) (n3)--(h3) (h1)--(h2) (h2)--(h3);
  \node[draw=none,rectangle,align=center] at (2.35,-2.0)
       {$A=\{a_1,a_2\}$\qquad $N_G(A)=\{u_1,u_2,u_3\}$};
  \node[draw=none,rectangle,font=\scriptsize,align=left] at (5.45,1.55)
       {blue: $A$\\orange: $N_G(A)$\\green: vertices of $G_A$};
  \node[draw=none,rectangle,font=\small] (galab) at (4.5,-0.8) {$G_A$};

\end{tikzpicture}
\caption{Localization at an independent set $A$.  The graph $G_A$ is obtained by deleting the closed neighborhood $N_G[A]=A\cup N_G(A)$ and keeping only the remaining induced subgraph.}
\label{fig:localization-example}
\end{figure}

\begin{definition}[$\W_p$ graphs]
Let $p\ge1$.  A graph $G$ belongs to $\W_p$ if $n(G)\ge p$ and, for every $p$ pairwise disjoint independent sets $A_1,\ldots,A_p$, there exist $p$ pairwise disjoint maximum independent sets $S_1,\ldots,S_p\in\Omega(G)$ such that $A_i\subseteq S_i$ for every $i$.
\end{definition}

\begin{definition}[Quasi-regularizability]
Let $\lambda>0$.  A graph $G$ is $\lambda$-quasi-regularizable if
\[
        \lambda |A|\le |N_G(A)|
\]
for every independent set $A$ of $G$.  The case $\lambda=1$ is the usual notion of quasi-regularizability due to Berge \cite{Berge}.
\end{definition}

\begin{definition}[Independence polynomial, log-concavity, unimodality]
Let
\[
        I(G;x)=\sum_{k=0}^{\alpha(G)}s_kx^k
        =s_0+s_1x+\cdots+s_{\alpha(G)}x^{\alpha(G)}
\]
be the independence polynomial of $G$, where $s_k$ is the number of independent sets of cardinality $k$.  Following \cite{HoangLevitMandrescuPham}, we say that $I(G;x)$ is
\begin{itemize}[leftmargin=2em]
\item \emph{log-concave} if
\[
        s_k^2\ge s_{k-1}s_{k+1}
        \qquad(1\le k\le\alpha(G)-1),
\]
\item \emph{unimodal} if there exists an index $0\le t\le\alpha(G)$ such that
\[
        s_0\le\cdots\le s_{t-1}\le s_t\ge s_{t+1}\ge\cdots\ge s_{\alpha(G)}.
\]
\end{itemize}
\end{definition}

\begin{lemma}[Keilson--Gerber product lemma]\label{lem:product}
Let $P(x)$ and $Q(x)$ have nonnegative coefficients and no internal zero coefficients.  If $P(x)$ is log-concave and $Q(x)$ is unimodal, then $P(x)Q(x)$ is unimodal.  If both $P(x)$ and $Q(x)$ are log-concave, then $P(x)Q(x)$ is log-concave.
\end{lemma}

The lemma is the standard polynomial form of a convolution theorem for nonnegative sequences.  Keilson and Gerber established the classical strong-unimodality principle behind the first assertion \cite{KeilsonGerber}; in the present nonnegative setting one must include the no-internal-zero condition, as in the corrected formulation of Levit and Mandrescu \cite[Theorem~1.2 and Corollary~1.3]{LevitMandrescuOperations}.  The compact independence-polynomial wording also appears explicitly in \cite[Theorem~1]{LevitMandrescuVeryWellCoveredLogConcave} and is the form quoted by Hoang--Levit--Mandrescu--Pham \cite[Lemma~1.4]{HoangLevitMandrescuPham}.  Since independence polynomials have strictly positive coefficients from degree $0$ through degree $\alpha(G)$, the hypothesis is automatic in all applications below.

\begin{lemma}[Standard facts on $\W_2$ graphs]\label{lem:w2facts}
Let $G\in\W_2$.
\begin{enumerate}[label=\textup{(\roman*)}]
\item $G$ has no isolated vertices.
\item If $A\in\Ind(G)$ and $|A|<\alpha(G)$, then $G_A\in\W_2$ and $\alpha(G_A)=\alpha(G)-|A|$.
\item A graph is in $\W_2$ if and only if each of its connected components is in $\W_2$.
\item If $A\in\Ind(G)$ is not maximum and $v\notin A$, then there is $S\in\Omega(G)$ such that $A\subseteq S$ and $v\notin S$.
\item If $C$ is connected, $C\in\W_2$, and $C\ne K_2$, then $C$ has no leaf.  In particular, if $C$ is connected, $C\in\W_2$, and $\alpha(C)>1$, then $\delta(C)\ge2$.
\end{enumerate}
\end{lemma}

\begin{proof}
For (i), if $z$ were isolated, then every maximum independent set would contain $z$, contradicting the existence of two disjoint maximum independent sets in the definition of $\W_2$ applied to the two empty independent sets.  The localization statement (ii) is the standard localization theorem for $\W_2$ graphs; see \cite[Corollary~2.3(i)]{LevitMandrescuRevisited} and also \cite[Lemma~2.7(iii)]{HoangLevitMandrescuPham}.  The identity for the independence number follows because a maximum independent set of $G_A$, together with $A$, is independent in $G$, and conversely any maximum independent set of $G$ containing $A$ leaves a maximum independent set in $G_A$.

For (iii), one may either check the defining extension property component by component, or quote \cite[Theorem~2.6]{HoangLevitMandrescuPham}.  Maximum independent sets in a disjoint union are exactly unions of maximum independent sets of the components, and restricting the defining property to one component gives the converse.

Assertion (iv) is the avoidance characterization \cite[Theorem~2.2(vii)]{LevitMandrescuRevisited}; it also follows directly from the definition of $\W_2$: apply it to the two disjoint independent sets $A$ and $\{v\}$.  We obtain disjoint maximum independent sets $S,T$ with $A\subseteq S$ and $v\in T$, hence $v\notin S$.

Finally, (v) is \cite[Corollary~2.3(ii)]{LevitMandrescuRevisited}.  For completeness we indicate the short argument.  Suppose that a connected $C\in\W_2$, $C\ne K_2$, has a leaf $x$ with unique neighbor $y$.  Since $C\ne K_2$, the vertex $y$ has another neighbor $z$.  By (iv), applied to the non-maximum independent set $\{z\}$ and the vertex $x$, there is a maximum independent set $S$ with $z\in S$ and $x\notin S$.  Then $y\notin S$, and hence $S\cup\{x\}$ is independent, contradicting the maximality of $S$.
\end{proof}

\begin{lemma}[Known coefficient inequalities]\label{lem:coef}
Let $G$ have order $n$, independence number $\alpha$, and independence polynomial $I(G;x)=\sum_{k=0}^{\alpha}s_kx^k$.
\begin{enumerate}[label=\textup{(\roman*)}]
\item If $G$ is $\lambda$-quasi-regularizable, then
\[
        (k+1)s_{k+1}\le \bigl(n-(\lambda+1)k\bigr)s_k,
        \qquad 0\le k\le \alpha-1.
\]
This is due to Levit and Mandrescu \cite[Theorem~2.1]{LevitMandrescuRoller}.
\item If $G$ is connected and $G\in\W_p$, then
\[
        p(\alpha-k)s_k\le(k+1)s_{k+1},
        \qquad 1\le k\le \alpha-1.
\]
This is due to Hoang, Levit, Mandrescu and Pham \cite[Theorem~2.6]{HoangLevitMandrescuPhamCliqueCorona}.
\end{enumerate}
\end{lemma}

These two inequalities are also recorded together as \cite[Lemma~3.2]{HoangLevitMandrescuPham}; we quote them here because they are the coefficient input for the criteria below.

\begin{corollary}\label{cor:components}
	Let $G$ be a graph with connected components $H_1,\ldots,H_q$.  If each
	$I(H_i;x)$ is log-concave, then $I(G;x)$ is log-concave.  If each
	$I(H_i;x)$ is unimodal and all but possibly one of them are log-concave,
	then $I(G;x)$ is unimodal.
\end{corollary}

\begin{proof}
	The independence polynomial is multiplicative over disjoint unions:
	\[
	I(G;x)=\prod_{i=1}^q I(H_i;x).
	\]
	The assertions follow from Lemma~\ref{lem:product}.
\end{proof}

\section{Coefficient criteria for Problem 4.1}\label{sec:coefficients}

The following criterion is the coefficient engine behind all log-concavity applications in this paper.  It separates the enumerative step from the structural question of when a $\W_p$ graph is $\lambda$-quasi-regularizable.

\begin{theorem}[General log-concavity criterion]\label{thm:LC-general}
Let $G$ be a connected graph in $\W_p$, with $n=n(G)$ and $\alpha=\alpha(G)$.  Suppose that $G$ is $\lambda$-quasi-regularizable.  If
\[
        \Phi_{p,\lambda,n,\alpha}(k)
        =(k+1)p(\alpha-k+1)-k\bigl(n-(\lambda+1)k\bigr)\ge0
\]
for every $1\le k\le\alpha-1$, then $I(G;x)$ is log-concave.
\end{theorem}

\begin{proof}
Apply Lemma~\ref{lem:coef}(i) with index $k$ and Lemma~\ref{lem:coef}(ii) with index $k-1$.  For $1\le k\le\alpha-1$,
\[
        s_{k+1}\le \frac{n-(\lambda+1)k}{k+1}s_k,
        \qquad
        s_{k-1}\le \frac{k}{p(\alpha-k+1)}s_k.
\]
Multiplying gives
\[
        s_{k-1}s_{k+1}
        \le
        \frac{k\bigl(n-(\lambda+1)k\bigr)}{(k+1)p(\alpha-k+1)}s_k^2.
\]
The hypothesis is exactly the assertion that the rational factor is at most $1$.
\end{proof}

\begin{corollary}[The $p$-quasi-regularizable specialization]\label{cor:pquasi}
Let $G$ be connected, $G\in\W_p$, with $n=n(G)$ and $\alpha=\alpha(G)$.  Suppose that $G$ is $p$-quasi-regularizable.  If
\[
        f_{p,n,\alpha}(k)=k^2-(n-p\alpha)k+p\alpha+p\ge0
\]
for every $1\le k\le\alpha-1$, then $I(G;x)$ is log-concave.
\end{corollary}

\begin{proof}
Put $\lambda=p$ in Theorem~\ref{thm:LC-general}.  The inequality
\[
        (k+1)p(\alpha-k+1)-k(n-(p+1)k)\ge0
\]
is equivalent to the displayed quadratic inequality.
\end{proof}

\begin{theorem}[Explicit interval form]\label{thm:interval}
Let $G$ be connected, $G\in\W_p$, with $n=n(G)$ and $\alpha=\alpha(G)\ge2$.  Suppose that $G$ is $p$-quasi-regularizable.  Then $I(G;x)$ is log-concave whenever either
\[
        (p+1)\alpha\le n\le p\alpha+2\sqrt{p\alpha+p},
        \qquad
        \frac{\alpha^2}{4(\alpha+1)}\le p,
\]
or
\[
        p\alpha+2\sqrt{p\alpha+p}<n\le
        \frac{(\alpha^2+1)p+(\alpha-1)^2}{\alpha-1},
        \qquad
        \frac{\alpha(\alpha-1)}{\alpha+1}\le p.
\]
\end{theorem}

\begin{proof}
By Corollary~\ref{cor:pquasi}, it is enough to guarantee
\[
        f(k)=k^2-(n-p\alpha)k+p\alpha+p\ge0
        \qquad(1\le k\le\alpha-1).
\]
Write $r=n-p\alpha$ and $c=p\alpha+p$.  The discriminant is
\[
        \Delta=r^2-4c.
\]
If $\Delta\le0$, then $f(k)\ge0$ for all real $k$.  This is equivalent to
\[
        n\le p\alpha+2\sqrt{p\alpha+p}.
\]
Because $p$-quasi-regularizability already implies $n\ge(p+1)\alpha$, the first interval is relevant exactly when
\[
        (p+1)\alpha\le p\alpha+2\sqrt{p\alpha+p},
\]
which is equivalent to $\alpha^2/(4(\alpha+1))\le p$.

Assume now that $\Delta>0$.  Let
\[
        \rho_1=\frac{r-\sqrt\Delta}{2},
        \qquad
        \rho_2=\frac{r+\sqrt\Delta}{2}
\]
be the two real roots.  Since $f$ opens upward, $f(k)\ge0$ for every integer $1\le k\le\alpha-1$ whenever $\alpha-1\le\rho_1$.  This condition is equivalent to
\[
        \sqrt\Delta\le r-2\alpha+2,                         \tag{1}
\]
where the right-hand side must be nonnegative.  Squaring (1) gives
\[
        r^2-4(p\alpha+p)\le (r-2\alpha+2)^2,
\]
or, after simplification,
\[
        n\le \frac{(\alpha^2+1)p+(\alpha-1)^2}{\alpha-1}.     \tag{2}
\]
Conversely, if $r-2\alpha+2\ge0$, then (2) implies (1), and hence $f(k)\ge0$ on the whole interval $1\le k\le\alpha-1$.

It remains only to check that the displayed lower bound in the statement forces the needed nonnegativity of $r-2\alpha+2$ in the integral range.  The hypothesis
\[
        \frac{\alpha(\alpha-1)}{\alpha+1}\le p
\]
implies
\[
        2\sqrt{p\alpha+p}\ge 2\alpha-2.
\]
Thus, since $n$ is an integer, the strict inequality
\[
        p\alpha+2\sqrt{p\alpha+p}<n
\]
forces $n\ge p\alpha+2\alpha-1$, and therefore $r-2\alpha+2\ge1$.  This is precisely the second case.  The calculation is the same discriminant analysis used in \cite[Theorem~3.3]{HoangLevitMandrescuPham}, but stated here independently of the special threshold theorem for $p\ne2$.
\end{proof}

\begin{theorem}[Unimodality criterion]\label{thm:unimodal}
Let $G$ be connected, $G\in\W_p$, $G$ $\lambda$-quasi-regularizable, with $n=n(G)$ and $\alpha=\alpha(G)$.  Define
\[
        L=\left\lfloor \frac{p\alpha-1}{p+1}\right\rfloor,
        \qquad
        R=\left\lceil \frac{n-1}{\lambda+2}\right\rceil.
\]
Then
\[
        s_0\le s_1\le\cdots\le s_{L+1},
        \qquad
        s_R\ge s_{R+1}\ge\cdots\ge s_\alpha.
\]
In particular, if $R\le L+1$, then $I(G;x)$ is unimodal.
\end{theorem}

\begin{proof}
Lemma~\ref{lem:coef}(ii) gives
\[
        \frac{s_{k+1}}{s_k}\ge \frac{p(\alpha-k)}{k+1}.
\]
The right-hand side is at least $1$ exactly when $k\le(p\alpha-1)/(p+1)$; this proves the increasing chain.  Lemma~\ref{lem:coef}(i) gives
\[
        \frac{s_{k+1}}{s_k}\le \frac{n-(\lambda+1)k}{k+1}.
\]
The right-hand side is at most $1$ exactly when $k\ge(n-1)/(\lambda+2)$; this proves the decreasing chain.  If the two chains overlap, the sequence has a peak interval and is unimodal.  Figure~\ref{fig:coeff-bounds} shows the mechanism.
\end{proof}

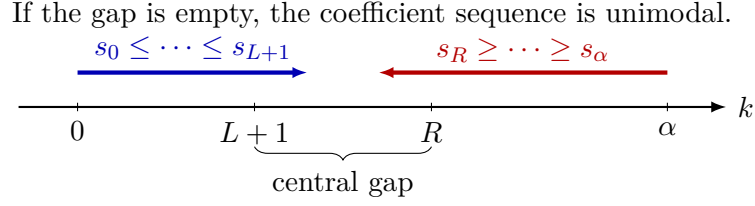
\begin{figure}[t]
\centering
\begin{tikzpicture}[x=0.78cm,y=0.7cm,font=\small]
  \draw[-{Latex[length=2mm]},thick] (0,0)--(12,0) node[right] {$k$};
  \foreach \x/\lab in {1/$0$,4/$L+1$,7/$R$,11/$\alpha$}{\draw (\x,0.08)--(\x,-0.08) node[below] {\lab};}
  \draw[blue!70!black,ultra thick,-{Latex[length=2mm]}] (1,0.65)--(4.9,0.65);
  \node[blue!70!black,above] at (2.95,0.65) {$s_0\le\cdots\le s_{L+1}$};
  \draw[red!70!black,ultra thick,{Latex[length=2mm]}-] (6.1,0.65)--(11,0.65);
  \node[red!70!black,above] at (8.55,0.65) {$s_R\ge\cdots\ge s_\alpha$};
  \draw[decorate,decoration={brace,amplitude=5pt,mirror}] (4,-0.75)--(7,-0.75) node[midway,below=5pt,align=center] {central gap};
  \node[align=center] at (6,1.75) {If the gap is empty, the coefficient sequence is unimodal.};
\end{tikzpicture}
\caption{The two-sided coefficient inequalities force an initial increasing segment and a final decreasing segment.  The overlap condition $R\le L+1$ yields unimodality.}
\label{fig:coeff-bounds}
\end{figure}

\section{A local expansion theorem for connected \texorpdfstring{$\W_2$}{W2} graphs}\label{sec:local}

This section proves the result that resolves Conjecture 4.2.  The proof is organized as a minimal-counterexample argument.  The forbidden normal form is drawn in Figure~\ref{fig:normal-form}.

\begin{figure}[t]
\centering
\begin{tikzpicture}[font=\small, every node/.style={align=center}]
  \tikzset{v/.style={circle,draw,minimum size=6.5mm,inner sep=0pt},
           priv/.style={circle,draw,fill=blue!12,minimum size=6.5mm,inner sep=0pt},
           shared/.style={circle,draw,fill=orange!18,minimum size=7mm,inner sep=0pt}}
  \node[v] (a1) at (0,0) {$x_1$};
  \node[v] (a2) at (1.5,0) {$x_2$};
  \node[v] (a3) at (3,0) {$x_3$};
  \node[draw=none] (adots) at (4.2,0) {$\cdots$};
  \node[v] (aa) at (5.3,0) {$x_a$};

  \node[priv] (m1) at (0,-1.55) {$m_1$};
  \node[priv] (m2) at (1.5,-1.55) {$m_2$};
  \node[priv] (m3) at (3,-1.55) {$m_3$};
  \node[draw=none] at (4.2,-1.55) {$\cdots$};
  \node[priv] (ma) at (5.3,-1.55) {$m_a$};

  \node[shared] (u1) at (1.0,-3.05) {$u$};
  \node[shared] (u2) at (3.8,-3.05) {$w$};
  \node[draw=none] at (5.2,-3.05) {$\cdots$};

  \node[ellipse,draw,minimum width=3.0cm,minimum height=1.2cm,fill=green!10] (H) at (7.7,-2.05) {$H=G_A$};

  \foreach \i in {1,2,3}{\draw[thick] (a\i)--(m\i);}
  \draw[thick] (aa)--(ma);
  \draw[thick,orange!70!black] (u1)--(a1);
  \draw[thick,orange!70!black] (u1)--(a2);
  \draw[thick,orange!70!black] (u1)--(a3);
  \draw[thick,orange!70!black] (u1)--(aa);
  \draw[thick,orange!70!black] (u2)--(a2);
  \draw[thick,orange!70!black] (u2)--(a3);
  \draw[thick,dashed] (u1)--(H);

  \begin{scope}[on background layer]
    \node[draw=blue!60!black,rounded corners,fit=(m1)(m2)(m3)(ma),inner sep=4pt,label=left:$M$] {};
    \node[draw=orange!70!black,rounded corners,fit=(u1)(u2),inner sep=8pt,label=left:$U$] {};
    \node[draw=black!55,rounded corners,fit=(a1)(a2)(a3)(aa),inner sep=5pt,label=left:$A$] {};
  \end{scope}
  \node[draw=none,font=\scriptsize] at (7.8,0.45) {$|D|=2a-1$};
\end{tikzpicture}
\caption{The normal form forced by a smallest hypothetical violation of $|N(A)|\ge2|A|$.  Private neighbors $m_i$ cannot see the localization $H=G_A$; any vertex of $U$ that sees $H$ is forced to be universal on $A$.}
\label{fig:normal-form}
\end{figure}
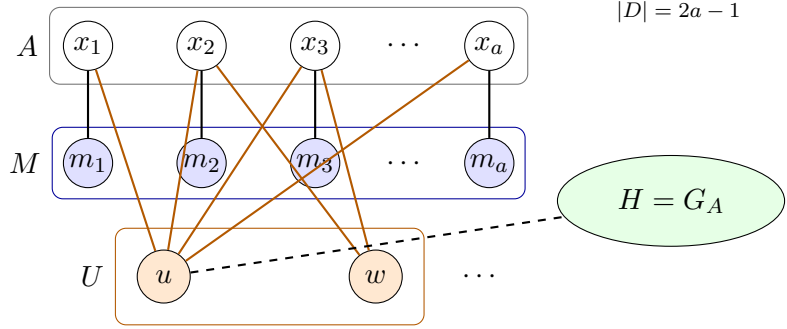

\begin{theorem}[Local neighborhood expansion]\label{thm:local-expansion}
Let $G$ be a connected graph in $\W_2$.  If $A\in\Ind(G)$ is not maximum, then
\[
        |N_G(A)|\ge 2|A|.
\]
\end{theorem}

\begin{proof}
Assume the theorem is false.  Choose a counterexample $(G,A)$ with $G$ connected, $G\in\W_2$, $A\in\Ind(G)$ non-maximum, and
\[
        |N_G(A)|<2|A|,
\]
so that $a=|A|$ is minimum among all such counterexamples.  Put
\[
        D=N_G(A),
        \qquad
        H=G_A,
        \qquad
        b=\alpha(H)=\alpha(G)-a.
\]
Since $A$ is not maximum, $b\ge1$ and $H$ is nonempty.

\smallskip
\noindent\textbf{Step 1: the deficiency is exactly one.}
We first show that
\[
        |D|=2a-1.
\]
The case $a=1$ is impossible: if $A=\{x\}$, then $\deg_G(x)<2$, while $x$ is not a maximum independent set.  Lemma~\ref{lem:w2facts}(i) rules out degree $0$, and since $G\ne K_2$, the leafless property in Lemma~\ref{lem:w2facts}(v) rules out degree $1$.  Hence $a\ge2$.

Fix $x\in A$ and define the private neighborhood of $x$ with respect to $A$ by
\[
        X_x=N_G(x)-N_G(A-\{x\}).
\]
The set $A-\{x\}$ is still non-maximum, because any vertex that extends $A$ also extends $A-\{x\}$.  By the minimality of $a$,
\[
        |N_G(A-\{x\})|\ge2(a-1).
\]
In the localization $G_{A-\{x\}}$, the vertex $x$ is present.  Since $G_{A-\{x\}}\in\W_2$, it has no isolated vertices; therefore $X_x\ne\varnothing$.  The disjoint union
\[
        N_G(A)=N_G(A-\{x\})\sqcup X_x
\]
then gives
\[
        |D|\ge2(a-1)+1=2a-1.
\]
Since $|D|<2a$, we have $|D|=2a-1$, and moreover
\[
        |X_x|=1\qquad\text{for every }x\in A.
\]
Let $m_x$ be the unique vertex of $X_x$.  Put
\[
        M=\{m_x:x\in A\},
        \qquad
        U=D-M.
\]
Then $|M|=a$, $|U|=a-1$, and every vertex of $U$ has at least two neighbors in $A$.

\smallskip
\noindent\textbf{Step 2: private neighbors do not see the localization.}
For each $x\in A$, the vertices remaining in $G_{A-\{x\}}$ include $x$, its unique neighbor $m_x$ in that localization, and $H$.  Thus $x$ is a leaf of the component of $G_{A-\{x\}}$ containing it.  By Lemma~\ref{lem:w2facts}(iii) the component is in $\W_2$, and by Lemma~\ref{lem:w2facts}(v) it must be $K_2$.  Hence
\[
        N_G(m_x)\cap V(H)=\varnothing
        \qquad(x\in A).
\]
Figure~\ref{fig:k2-localization} illustrates this forced $K_2$ component.

\begin{figure}[t]
\centering
\begin{tikzpicture}[font=\small, every node/.style={align=center}]
  \tikzset{v/.style={circle,draw,minimum size=7mm,inner sep=0pt},
           gone/.style={circle,draw,dashed,minimum size=7mm,inner sep=0pt,text=gray}}
  \node[v,fill=white] (x) at (0,0) {$x$};
  \node[v,fill=blue!12] (m) at (1.4,0) {$m_x$};
  \draw[thick] (x)--(m);
  \node[ellipse,draw,fill=green!10,minimum width=2.5cm,minimum height=1.0cm] (H) at (4.2,0) {$H=G_A$};
  \draw[dashed,red!70!black,thick] (m)--(H) node[midway,above] {forbidden};
  \node[draw=none,align=center] at (2.1,-1.15) {in $G_{A-\{x\}}$, $x$ is a leaf};
  \node[draw=none,align=center] at (2.1,-1.65) {a connected $\W_2$ component with a leaf is $K_2$};
\end{tikzpicture}
\caption{Localizing at $A-\{x\}$ leaves $x$ with the unique neighbor $m_x$.  The leafless property of nontrivial connected $\W_2$ graphs forces the whole component to be $K_2$, so $m_x$ has no neighbor in $H$.}
\label{fig:k2-localization}
\end{figure}

\smallskip
\noindent\textbf{Step 3: every non-private neighbor that touches $H$ is universal on $A$.}
Let $u\in U$ and assume $u$ has a neighbor in $H$.  Set
\[
        R=N_G(u)\cap A,
        \qquad
        T=A-R.
\]
Since $u\in U$, we have $|R|\ge2$.  We prove that $R=A$.

Suppose, to the contrary, that $T\ne\varnothing$.  Consider the localization $G_T$; see Figure~\ref{fig:step3-localization}.  The set $R$, the vertex $u$, and the part of $H$ touched by $u$ lie in one connected component $C$ of $G_T$.  This component belongs to $\W_2$ by Lemma~\ref{lem:w2facts}(ii),(iii).  The set $R$ is not maximum in $C$, because a neighbor of $u$ lying in $H$ can be added to $R$.

Let
\[
        t(R)=\bigl|\{w\in U: N_G(w)\cap A\subseteq R\}\bigr|.
\]
Inside the component $C$, the neighbors of $R$ are precisely the private vertices $m_x$ with $x\in R$, together with the $t(R)$ vertices of $U$ not deleted by the localization at $T$.  Therefore
\[
        |N_C(R)|=|R|+t(R).
\]
If $t(R)<|R|$, then $C$ would be a connected $\W_2$ graph with a non-maximum independent set $R$ satisfying $|N_C(R)|<2|R|$, contradicting the minimality of $a$ because $|R|<a$.  Hence
\[
        t(R)\ge |R|.        \tag{1}
\]

On the other hand, the set $T$ is non-maximum in $G$, and $|T|<a$.  Minimality gives
\[
        |N_G(T)|\ge2|T|.
\]
The neighborhood of $T$ consists of the $|T|$ private vertices $m_x$ with $x\in T$, together with the vertices of $U$ that see $T$.  Thus at least $|T|$ vertices of $U$ see $T$.  These are exactly the vertices of $U$ not counted by $t(R)$.  Since $|U|=a-1$, we get
\[
        t(R)\le(a-1)-|T|=|R|-1,
\]
contradicting (1).  Therefore $R=A$, and $u$ is universal on $A$.

\begin{figure}[t]
\centering
\begin{tikzpicture}[font=\small]
  \tikzset{v/.style={circle,draw,minimum size=7mm,inner sep=0pt},
           del/.style={circle,draw,dashed,minimum size=7mm,inner sep=0pt,text=gray},
           keep/.style={circle,draw,fill=green!10,minimum size=7mm,inner sep=0pt}}
  \node[v] (r1) at (0,1.15) {$R$};
  \node[del] (t1) at (-0.15,-1.55) {$T$};
  \node[v,fill=orange!18] (u) at (2.0,0.15) {$u$};
  \node[v,fill=blue!12] (mr) at (2.0,1.45) {$M_R$};
  \node[del] (ut) at (2.1,-1.8) {$U_T$};
  \node[keep] (ur) at (3.95,1.15) {$U_R$};
  \node[ellipse,draw,fill=green!8,minimum width=2.5cm,minimum height=1.2cm] (H) at (5.75,0.15) {$H$};
  \draw[thick] (r1)--(u);
  \draw[thick] (r1)--(mr);
  \draw[thick] (u)--(ur);
  \draw[thick] (u)--(H);
  \draw[dashed] (t1)--(ut);
  \draw[dashed] (t1)--(u);
  \node[draw=none,rectangle,align=center,font=\scriptsize] at (2.95,-2.7)
      {In $G_T$, the vertices of $T$ and all their neighbors are deleted.\\Only $R$, its private neighbors, the vertices of $U$ seeing only $R$, and $H$ remain.};
  \begin{scope}[on background layer]
    \node[draw=green!50!black,rounded corners,fill=green!4,fit=(r1)(u)(mr)(ur)(H),inner sep=5pt] {};
  \end{scope}
  \node[draw=none,rectangle,font=\small] at (3.0,2.4) {component $C$ of $G_T$};
\end{tikzpicture}
\caption{The localization $G_T$ used in Step~3.  If $u\in U$ sees $H$ but is not universal on $A$, then $A=R\sqcup T$ with $T\neq\varnothing$.  In $G_T$, the component $C$ contains $R$, the vertex $u$, the private neighbors of vertices in $R$, the vertices of $U$ whose neighbors in $A$ are contained in $R$, and the part of $H$ touched by $u$.}
\label{fig:step3-localization}
\end{figure}

Because $G$ is connected, $H$ is nonempty, and no private vertex $m_x$ sees $H$, at least one vertex of $U$ must see $H$.  Hence there exists a vertex $u\in U$ universal on $A$ and adjacent to $H$.

\smallskip
\noindent\textbf{Step 4: the universal vertex is unique.}
Suppose that two distinct vertices $u,v\in U$ are universal on $A$.  Apply the $\W_2$ property to the disjoint independent sets $\{u\}$ and $\{v\}$.  We obtain two disjoint maximum independent sets $S_u,S_v$ with $u\in S_u$ and $v\in S_v$.  Since $u$ and $v$ are adjacent to every vertex of $A$, neither maximum independent set can use any vertex of $A$.  Each of them contains at most $b$ vertices in $H$, so each must contain at least $a$ vertices from $D=N_G(A)$.  The two maximum independent sets are disjoint, so this would require at least $2a$ distinct vertices in $D$, contradicting $|D|=2a-1$.  Thus the universal vertex in $U$ is unique.  Fix it and call it $u$, and choose
\[
        q\in N_H(u).
\]

\smallskip
\noindent\textbf{Step 5: the final localization.}
Fix an arbitrary vertex $x\in A$ and set $R=A-\{x\}$.  In the localization $G_x$, the universal vertex $u$ is deleted.  By Step 3, every vertex of $U-\{u\}$ has no neighbor in $H$; by Step 2, the same is true of every private vertex.  Hence $G_x$ is the disjoint union of $H$ and a graph $L_R$ on the remaining vertices outside $H$:
\[
        V(L_R)=R\cup\{m_y:y\in R\}
        \cup\{w\in U-\{u\}:N_G(w)\cap A\subseteq R\}.
\]
Since $G_x\in\W_2$, also $L_R\in\W_2$.  Moreover,
\[
        \alpha(L_R)=\alpha(G_x)-\alpha(H)=(\alpha(G)-1)-b=a-1.
\]
Thus $R$ is a maximum independent set of $L_R$.

Apply the definition of $\W_2$ in $L_R$ to the two disjoint independent sets $R$ and $\varnothing$.  Since $R$ is already maximum, there exists a maximum independent set $C$ of $L_R$ disjoint from $R$.  Therefore
\[
        |C|=a-1,
        \qquad
        C\subseteq D-\{u\}.
\]
As $C$ is maximal in $L_R$, it dominates $R$.

We claim that $C$ also dominates every vertex of $(D-\{u\})-C$.  It already dominates all vertices of $L_R-C$ by maximality in $L_R$.  Suppose that some vertex
\[
        d\in (D-\{u\})-V(L_R)
\]
has no neighbor in $C$.  Then $B=C\cup\{d\}$ is independent.  The vertices outside $L_R$ in $D-\{u\}$ are exactly the non-universal neighbors deleted when localizing at $x$, together with $m_x$; in particular, $d$ is adjacent to $x$.  Since $C$ dominates $R=A-\{x\}$, the independent set $B$ sees every vertex of $A$.

Apply $\W_2$ in $G$ to the disjoint independent sets $\{u\}$ and $B$.  A maximum independent set containing $u$ cannot contain vertices of $A$; a maximum independent set containing $B$ also cannot contain vertices of $A$, because $B$ sees all of $A$.  Each of the two disjoint maximum independent sets must therefore use at least $a$ vertices of $D$, again contradicting $|D|=2a-1$.  This proves that $C$ dominates all of $(D-\{u\})-C$, as claimed.

Now $C\cup\{q\}$ is independent: vertices of $C$ are either private vertices or non-universal vertices of $U$, and none of these has a neighbor in $H$.  Also $|C\cup\{q\}|=a<\alpha(G)=a+b$, so it is not maximum.  By the avoidance property in Lemma~\ref{lem:w2facts}(iv), there is a maximum independent set $S$ of $G$ such that
\[
        C\cup\{q\}\subseteq S,
        \qquad
        x\notin S.
\]
This is impossible; see Figure~\ref{fig:final-obstruction}.  The vertex $q$ forbids $u$, the set $C$ dominates $R$ and all of $(D-\{u\})-C$, and $x$ is explicitly forbidden.  Hence outside $H$ the set $S$ can contain no vertex except the vertices of $C$.  Inside $H$ it contains at most $b$ vertices.  Consequently
\[
        |S|\le |C|+b=(a-1)+b=\alpha(G)-1,
\]
contradicting $S\in\Omega(G)$.  This contradiction completes the proof.
\end{proof}

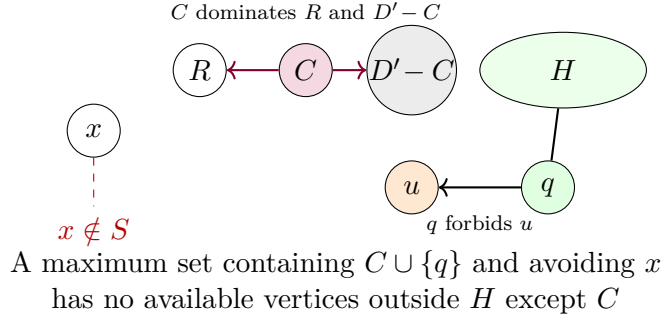
\begin{figure}[t]
\centering
\begin{tikzpicture}[font=\small]
  \tikzset{v/.style={circle,draw,minimum size=7mm,inner sep=0pt},
           cnode/.style={circle,draw,fill=purple!15,minimum size=7mm,inner sep=0pt},
           blocked/.style={circle,draw,fill=gray!15,minimum size=7mm,inner sep=0pt}}
  \node[v] (x) at (0,0) {$x$};
  \node[v] (r1) at (1.4,0.8) {$R$};
  \node[cnode] (c1) at (2.8,0.8) {$C$};
  \node[blocked] (d) at (4.2,0.8) {$D'\!-C$};
  \node[v,fill=orange!18] (u) at (4.2,-0.75) {$u$};
  \node[v,fill=green!12] (q) at (6.0,-0.75) {$q$};
  \node[ellipse,draw,fill=green!8,minimum width=2.2cm,minimum height=1.0cm] (H) at (6.2,0.8) {$H$};
  \draw[thick,purple!70!black,->] (c1)--(r1);
  \draw[thick,purple!70!black,->] (c1)--(d);
  \draw[thick,->] (q)--(u);
  \node[draw=none,align=center,font=\scriptsize] at (2.8,1.57) {$C$ dominates $R$ and $D'\!-C$};
  \node[draw=none,font=\scriptsize] at (5.1,-1.25) {$q$ forbids $u$};
  \draw[thick] (q)--(H);
  \draw[dashed,red!70!black] (x)--(0,-1.0) node[below] {$x\notin S$};
  \node[align=center] at (3.2,-2.0) {A maximum set containing $C\cup\{q\}$ and avoiding $x$\\has no available vertices outside $H$ except $C$};
\end{tikzpicture}
\caption{The final contradiction.  After localizing at $x$, a maximum set $C$ disjoint from $R=A-\{x\}$ dominates all remaining vertices outside $H$ except the universal vertex $u$, while $q\in H$ forbids $u$.}
\label{fig:final-obstruction}
\end{figure}

\section{Solution of Conjecture 4.2}\label{sec:conjecture}

\begin{theorem}[Conjecture 4.2]\label{thm:conj42}
Let $G$ be a connected graph in $\W_2$.  Then $G$ is $2$-quasi-regularizable if and only if
\[
        n(G)\ge3\alpha(G).
\]
\end{theorem}

\begin{proof}
If $G$ is $2$-quasi-regularizable and $S\in\Omega(G)$, then
\[
        n(G)-\alpha(G)=|N_G(S)|\ge2|S|=2\alpha(G),
\]
so $n(G)\ge3\alpha(G)$.

Conversely, assume $n(G)\ge3\alpha(G)$ and let $A\in\Ind(G)$.  If $A$ is not maximum, then Theorem~\ref{thm:local-expansion} gives
\[
        |N_G(A)|\ge2|A|.
\]
If $A$ is maximum, then
\[
        |N_G(A)|=n(G)-\alpha(G)\ge2\alpha(G)=2|A|.
\]
Thus $|N_G(A)|\ge2|A|$ for every independent set $A$, which is precisely $2$-quasi-regularizability.
\end{proof}

\begin{remark}
Theorem~\ref{thm:local-expansion} is stronger than Conjecture 4.2.  The global inequality $n(G)\ge3\alpha(G)$ is needed only for maximum independent sets.  Every non-maximum independent set in a connected $\W_2$ graph satisfies the $2$-quasi-regularizability inequality automatically.
\end{remark}

\section{Consequences for Problem 4.1}\label{sec:consequences}

We now insert the solved $p=2$ threshold into the coefficient criteria of Section~\ref{sec:coefficients}.

\begin{corollary}[Log-concavity for connected $\W_2$ graphs]\label{cor:p2LC}
Let $G$ be connected, $G\in\W_2$, with $n=n(G)$ and $\alpha=\alpha(G)$.  Suppose that $n\ge3\alpha$.  If
\[
        k^2-(n-2\alpha)k+2\alpha+2\ge0
        \qquad(1\le k\le\alpha-1),
\]
then $I(G;x)$ is log-concave.
\end{corollary}

\begin{proof}
By Theorem~\ref{thm:conj42}, the graph $G$ is $2$-quasi-regularizable.  Apply Corollary~\ref{cor:pquasi} with $p=2$.
\end{proof}

\begin{corollary}[Explicit $p=2$ interval form]\label{cor:p2interval}
Let $G$ be connected, $G\in\W_2$, with $n=n(G)$ and $\alpha=\alpha(G)\ge2$.  Suppose that $n\ge3\alpha$.  Then $I(G;x)$ is log-concave whenever either
\[
        3\alpha\le n\le 2\alpha+2\sqrt{2\alpha+2},
        \qquad
        \frac{\alpha^2}{4(\alpha+1)}\le2,
\]
or
\[
        2\alpha+2\sqrt{2\alpha+2}<n\le
        \frac{2(\alpha^2+1)+(\alpha-1)^2}{\alpha-1},
        \qquad
        \frac{\alpha(\alpha-1)}{\alpha+1}\le2.
\]
\end{corollary}

\begin{proof}
Theorem~\ref{thm:conj42} gives $2$-quasi-regularizability from $n\ge3\alpha$.  The result is Theorem~\ref{thm:interval} with $p=2$.
\end{proof}

\begin{corollary}[Unimodality for connected $\W_2$ graphs]\label{cor:p2unimodal}
Let $G$ be connected, $G\in\W_2$, with $n=n(G)$ and $\alpha=\alpha(G)$.  Suppose that $n\ge3\alpha$.  If
\[
        \left\lceil \frac{n-1}{4}\right\rceil
        \le
        \left\lfloor \frac{2\alpha-1}{3}\right\rfloor+1,
\]
then $I(G;x)$ is unimodal.
\end{corollary}

\begin{proof}
By Theorem~\ref{thm:conj42}, the graph is $2$-quasi-regularizable.  Apply Theorem~\ref{thm:unimodal} with $p=\lambda=2$.
\end{proof}

\begin{corollary}[A unified set of sufficient conditions for Problem 4.1]\label{cor:problem41}
Let $G$ be a connected graph in $\W_p$, with $n=n(G)$ and $\alpha=\alpha(G)$.
\begin{enumerate}[label=\textup{(\alph*)}]
\item If $G$ is $\lambda$-quasi-regularizable and
\[
        (k+1)p(\alpha-k+1)-k\bigl(n-(\lambda+1)k\bigr)\ge0
\]
for every $1\le k\le\alpha-1$, then $I(G;x)$ is log-concave.
\item If $G$ is $\lambda$-quasi-regularizable and
\[
        \left\lceil \frac{n-1}{\lambda+2}\right\rceil
        \le
        \left\lfloor \frac{p\alpha-1}{p+1}\right\rfloor+1,
\]
then $I(G;x)$ is unimodal.
\item If $p\ne2$, then the theorem of Hoang--Levit--Mandrescu--Pham says that $n\ge(p+1)\alpha$ is equivalent to $p$-quasi-regularizability for connected $\W_p$ graphs.  Therefore parts \textup{(a)} and \textup{(b)} apply with $\lambda=p$ whenever $n\ge(p+1)\alpha$.
\item If $p=2$, then Theorem~\ref{thm:conj42} says that $n\ge3\alpha$ is equivalent to $2$-quasi-regularizability for connected $\W_2$ graphs.  Therefore parts \textup{(a)} and \textup{(b)} apply with $\lambda=2$ whenever $n\ge3\alpha$.
\end{enumerate}
\end{corollary}

\begin{remark}
Part (d) is the missing $p=2$ counterpart of the structural input used in \cite{HoangLevitMandrescuPham}.  The coefficient inequalities themselves are uniform in $p$; the exceptional nature of $p=2$ lies in proving the exact quasi-regularizability threshold.
\end{remark}

\section{Final remarks}

The local theorem explains why the case $p=2$ was delicate.  In a smallest hypothetical counterexample, localizing at $A-\{x\}$ creates a $K_2$ component $\{x,m_x\}$.  Such a component is allowed in $\W_2$, whereas higher $p$ arguments usually benefit from stronger local degree constraints.  The proof above eliminates precisely this obstruction by showing that the remaining connection to $G_A$ must pass through a unique universal vertex, and then using the $\W_2$ extension property one final time.

For independence polynomials, the conclusion is that the conditions in Problem 4.1 can now be applied uniformly at the quasi-regularizability threshold: for $p\ne2$ by the theorem of Hoang--Levit--Mandrescu--Pham, and for $p=2$ by Theorem~\ref{thm:conj42}.  Componentwise versions follow from Lemma~\ref{lem:product}, which is why log-concavity results for connected components extend cleanly to disconnected $\W_p$ graphs.

\medskip
\noindent\textbf{Acknowledgment.}
The authors thank the referee for the careful reading and helpful suggestions.

\section*{Declaration of generative AI and AI-assisted technologies in the writing process}
During the preparation of this work the authors used ChatGPT-3.5 in order to improve the grammar of several paragraphs of the text. After using this service, the authors reviewed and edited the content as needed and take full responsibility for the content of the publication.

\section*{Data availability}
Data sharing is not applicable to this article, as no datasets were generated or analyzed during the current study.

\section*{Declarations}
\noindent\textbf{Conflict of interest.} The authors declare that they have no conflict of interest.

\end{document}